\documentclass[envcountsect]{svmult}
\usepackage{mathptmx}       
\usepackage{helvet}         
\usepackage{courier}        
\usepackage{type1cm}        
\usepackage[colorlinks=true,allcolors=blue]{hyperref}
\usepackage{amsmath}
\usepackage{amssymb}
\usepackage{makeidx}         
\usepackage{graphicx}        
\usepackage{multicol}        
\usepackage[bottom]{footmisc}

\newcommand{\mb}[1]{\mathbb{#1}}
\newcommand{\mf}[1]{\mathbf{#1}}
\newcommand{\oE}{\mf{\operatorname{E}}}
\newcommand{\G}{\mb{G}}
\newcommand{\R}{\mb{R}}
\newcommand{\Z}{\mb{Z}}
\newcommand{\Rep}{P}
\newcommand{\V}[1]{\underline{#1}}
\newcommand{\abs}[1]{\left| #1 \right|}
\newcommand{\To}{\rightarrow}
\spnewtheorem{Lemma}{Lemma}[section]{\bfseries}{}
\spnewtheorem{Question}{Question}[section]{\itshape}{}
\spnewtheorem{Remark}{Remark}[section]{\itshape}{}
\newenvironment{proofof}[1]{\noindent\textbf{Proof of #1.}}{\qed}
\makeindex

\begin{document}

\title*{Twin bent functions and Clifford algebras}

\author{Paul~C.~Leopardi}
\authorrunning{P.~C.~Leopardi}
\institute{Mathematical Sciences Institute, The Australian National University.
\email{paul.leopardi@gmail.com}}

\date{Submitted: 15 September 2014, \\ Revised: 11 January 2015}

\maketitle

\abstract{
%
This paper examines a pair of bent functions on $\Z_2^{2m}$ and their relationship to
a necessary condition for the existence of an automorphism of an edge-coloured graph whose colours are defined
by the properties of a canonical basis for the real representation of the Clifford algebra $\R_{m,m}.$
Some other necessary conditions are also briefly examined.
%
}
\section{Introduction}
\label{sec-intro}
A recent paper \cite{Leo14Constructions} constructs a sequence of 
edge-coloured graphs
\index{edge-coloured graph} 
$\varDelta_m$ $(m \geqslant 1)$
with two edge colours, and makes the conjecture that for $m \geqslant 1,$ 
there is an automorphism of $\varDelta_m$ that swaps the two edge colours.
This conjecture can be refined into the following question.

\begin{Question}
\label{Question-1}
Consider the sequence of edge-coloured graphs $\varDelta_m$ $(m \geqslant 1)$ as defined in \cite{Leo14Constructions},
each with red subgraph $\varDelta_m[-1],$ and blue subgraph $\varDelta_m[1].$
For which $m \geqslant 1$ is there an automorphism of $\varDelta_m$ that swaps the subgraphs $\varDelta_m[-1]$ and $\varDelta_m[1]$?
\end{Question}

Note that the existence of such an automorphism automatically implies that the subgraphs $\varDelta_m[-1]$ and $\varDelta_m[1]$
are isomorphic.

Considering that it is known that $\varDelta_m[-1]$ is a 
strongly regular graph, 
\index{strongly regular graph}
a more general question can be asked concerning such graphs.
\newpage
First, we recall the relevant definition.

\begin{definition}\label{definition-strongly-regular-graph}  \cite{Bos63,BroCN89,Sei79}.
A simple graph $\Gamma$ of order $v$ is \emph{strongly regular} with parameters $(v,k,\lambda,\mu)$
if
\begin{itemize}
 \item 
each vertex has degree $k,$ 
 \item 
each adjacent pair of vertices has $\lambda$ common neighbours, and
\item
each nonadjacent pair of vertices has $\mu$ common neighbours.
\end{itemize}
\end{definition}

Now, the more general question.
\begin{Question}
\label{Question-2}
For which parameters $(v,k,\lambda,\mu)$ is there a an edge-coloured graph $\Gamma$ on $v$ vertices,
with two edge colours, red (with subgraph $\Gamma[-1]$) and blue (with subgraph $\Gamma[1]$), such that
the subgraph $\Gamma[-1]$ is a strongly regular graph with parameters $(v,k,\lambda,\mu),$
and such that there exists an automorphism of $\Gamma$ that swaps $\Gamma[-1]$ with $\Gamma[1]$?
\end{Question}
\begin{Remark}
Since the existence of such an automorphism implies that $\Gamma[-1]$ and $\Gamma[1]$ are isomorphic,
this implies that $\Gamma[1]$ is also a strongly regular graph with the same parameters as $\Gamma[-1].$

Questions~\ref{Question-1} and~\ref{Question-2}  were asked (in a slightly different form) at the workshop on 
``Algebraic design theory with Hada\-mard matrices'' in Banff in July 2014.
\end{Remark}

Further generalization gives the following questions.
\begin{Question}
\label{Question-3}
Given a positive integer $c>1,$ for what parameters $(v,k,\lambda,\mu)$ 
does there exist a $c k$ regular graph on $v$ vertices that can be given an edge colouring with $c$ colours, 
such that the edges corresponding to each color form a $(v,k,\lambda,\mu)$ strongly regular graph?

For what parameters is the $c$-edge-coloured $c k$ regular graph unique up to isomorphism? 
\end{Question}
\begin{Remark}
This question appears on MathOverflow \cite{Leo14MathOverflow}, and is partially answered
by Dima Pasechnik and Padraig \'O Cath\'ain,
specifically for the case where the $c k$ regular graph is the complete graph on $v = c k + 1$ vertices.
See the relevant papers by van Dam \cite{vanD03}, van Dam and Muzychuk \cite{vanDM10}, and \'O Cath{\'a}in \cite{OCat13}.
These partial answers do not apply to the specific case of Question~\ref{Question-1} because the graph $\varDelta_m$
is not a complete graph when $m>1.$
\end{Remark}

\begin{Question}
\label{Question-4}
For which parameters $(v,k,\lambda,\mu)$ does the edge-coloured graph $\Gamma$ from Question~\ref{Question-3}
have an automorphism that permutes the corresponding strongly regular subgraphs?
Which finite groups occur as permutation groups in this manner 
(i.e. as the group of permutations of strongly regular subgraphs of such an edge-coloured graph)? 
\end{Question}

This paper examines some of the necessary conditions for the graph $\varDelta_m$ to have an automorphism as per Question~\ref{Question-1}.
Questions~\ref{Question-2} to~\ref{Question-4} remain open for future investigation.

Considering that $\varDelta_m[-1]$ is a strongly regular graph, the first necessary condition is that
$\varDelta_m[1]$ is also a strongly regular graph, with the same parameters.
This is proven as Theorem~\ref{theorem-twins-are-strongly-regular} in Section~\ref{sec-strongly-regular}.
Some other necessary conditions are addressed in Section~\ref{sec-other}.

\section{A signed group and its real monomial representation}
\label{sec-first-bent}
The following definitions and results appear in the paper on Hada\-mard matrices and \index{Clifford algebras} \cite{Leo14Constructions},
and are presented here for completeness, since they are used below. 
Further details and proofs can be found in that paper, unless otherwise noted.

The signed group
$\G_{p,q}$ of order $2^{1+p+q}$ 
is extension of $\Z_2$ by $\Z_2^{p+q}$,
defined by the signed group presentation
\begin{align*}
\G_{p,q} := \bigg\langle \ 
&\mf{e}_{\{k\}}\ (k \in S_{p,q})\ \mid
\\
&\mf{e}_{\{k\}}^2 = -1\ (k < 0), \quad \mf{e}_{\{k\}}^2 = 1\ (k > 0),
\\
&\mf{e}_{\{j\}}\mf{e}_{\{k\}} = -\mf{e}_{\{k\}}\mf{e}_{\{j\}}\ (j \neq k) \bigg\rangle,
\end{align*}
where $S_{p,q} := \{-q,\ldots,-1,1,\ldots,p\}.$

The following construction of the real monomial representation $\Rep(\G_{m,m})$
of the group $\G_{m,m}$ is used in \cite{Leo14Constructions}.

The $2 \times 2$ orthogonal matrices
\begin{align*}
\oE_1 :=
\left[
\begin{array}{cc}
. & - \\
1 & .
\end{array}
\right],
\quad
\oE_2 :=
\left[
\begin{array}{cc}
. & 1 \\
1 & .
\end{array}
\right]
\end{align*}
generate $\Rep(\G_{1,1}),$ the real monomial representation of group $\G_{1,1}.$
The cosets of $\{\pm I\} \equiv \Z_2$ in $\Rep(\G_{1,1})$ are
ordered using a pair of bits, as follows.
\begin{align*}
0 &\leftrightarrow 00 \leftrightarrow \{ \pm I \},
\\
1 &\leftrightarrow 01 \leftrightarrow \{ \pm \oE_1 \},
\\
2 &\leftrightarrow 10 \leftrightarrow \{ \pm \oE_2 \},
\\
3 &\leftrightarrow 11 \leftrightarrow \{ \pm \oE_1 \oE_2 \}.
\end{align*}

For $m > 1$,
the real monomial representation $\Rep(\G_{m,m})$ of the 
group $\G_{m,m}$ consists of matrices of the form $G_1 \otimes G_{m-1}$
with $G_1$ in $\Rep(\G_{1,1})$ and $G_{m-1}$ in $\Rep(\G_{m-1,m-1}).$
The cosets of $\{\pm I\} \equiv \Z_2$ in $\Rep(\G_{m,m})$ are
ordered by concatenation of pairs of bits, 
where each pair of bits uses the ordering as per $\Rep(\G_{1,1}),$
and the pairs are ordered as follows.
\begin{align*}
0 &\leftrightarrow 00 \ldots 00 \leftrightarrow \{ \pm I \},
\\
1 &\leftrightarrow 00 \ldots 01 \leftrightarrow \{ \pm I_{(2)}^{\otimes {(m-1)}} \otimes  \oE_1 \},
\\
2 &\leftrightarrow 00 \ldots 10 \leftrightarrow \{ \pm I_{(2)}^{\otimes {(m-1)}} \otimes  \oE_2 \},
\\
&\ldots
\\
2^{2m} - 1 &\leftrightarrow 11 \ldots 11 \leftrightarrow \{ \pm (\oE_1 \oE_2)^{\otimes {m}} \}.
\end{align*}
(Here $I_{(2)}$ is used to distinguish 
this $2 \times 2$ unit matrix from the $2^m \times 2^m$ unit matrix $I$.)
In this paper, this ordering is called 
the \emph{Kronecker product ordering} of the cosets of $\{\pm I\}$ in $\Rep(\G_{m,m}).$

We recall here a number of well-known properties of the representation
$\Rep(\G_{m,m}).$
\begin{Lemma}\label{Lemma-representation-properties}
The group $\G_{m,m}$ and its real monomial representation $\Rep(\G_{m,m})$ 
satisfy the following properties.
\begin{enumerate}
\item 
Pairs of elements of $\G_{m,m}$ (and therefore $\Rep(\G_{m,m})$) either commute or anti\-commute:
for $g, h \in \G_{m,m},$ either $h g = g h$ or $h g = - g h.$
\item
The matrices $E \in \Rep(\G_{m,m})$ are orthogonal: $E E^T = E^T E = I.$
\item
The matrices $E \in \Rep(\G_{m,m})$ are either symmetric and square to give $I$ or 
skew and square to give $-I$: either $E^T = E$ and $E^2 =I$ or $E^T = -E$ and $E^2 = -I.$
\end{enumerate}
\end{Lemma}
%

Taking the positive signed element of each of the $2^{2m}$ cosets listed above
defines a transversal of $\{\pm I\}$ in $\Rep(\G_{m,m})$
which is also a monomial basis for the real representation of the Clifford algebra $\R_{m,m}$ in 
Kronecker product order.
In this paper, we call this ordered monomial basis the \emph{positive signed basis} of $\Rep(\R_{m,m}).$ 
For example,
$(I, \oE_1, \oE_2, \oE_1 \oE_2)$
is the positive signed basis of $\Rep(\R_{1,1}).$
Note: any other choice of signs will give a different transversal of $\{\pm I\}$ in $\Rep(\G_{m,m}),$
and hence an equivalent ordered monomial basis of $\Rep(\R_{m,m}),$ but we choose positive signs here for definiteness.

\begin{definition}\label{definition-gamma}
We define the function $\gamma_m : \Z_{2^{2 m}} \To \Rep(\G_{m,m})$ 
to choose the corresponding basis matrix from the positive signed basis of $\Rep(\R_{m,m}),$
using the Kronecker product ordering.
This ordering also defines a corresponding function on $\Z_2^{2 m},$
which we also call $\gamma_m.$
\end{definition}
For example,
\begin{align*}
\begin{array}{ll}
\gamma_1(0) = \gamma_1(00) = I,
&\gamma_1(1) = \gamma_1(01) = \oE_1,
\\
\gamma_1(2) = \gamma_1(10) = \oE_2,
\quad
&\gamma_1(3) = \gamma_1(11) = \oE_1 \oE_2.
\end{array}
\end{align*}

\section{Two bent functions}
\label{sec-second-bent}

We now define two functions, $\sigma_m$ and $\tau_m$ on $\Z_2^{2 m},$
and show that both of these are bent.
First, recall the relevant definition.

\begin{definition}\label{definition-bent-function} \cite[p. 74]{Dil74}.

A Boolean function $f : \Z_2^m \To \Z_2$ is \emph{bent}
if its Hada\-mard transform has constant magnitude.
Specifically:
\begin{enumerate}
 \item 
The Sylvester Hada\-mard matrix $H_m,$ of order $2^m,$
is defined by
\begin{align*}
H_1 &:=
\left[
\begin{array}{cc}
1 & 1 \\
1 & -
\end{array}
\right],
\\
H_m &:= H_{m-1} \otimes H_1, \quad \text{for} \quad m > 1.
\end{align*}

\item
For a Boolean function $f : \Z_2^m \To \Z_2,$
define the vector $\V{f}$ by
\begin{align*}
\V{f} &:= [(-1)^{f[0]}, (-1)^{f[1]}, \ldots, (-1)^{f[2^m-1]}]^T, 
\end{align*}
where the value of $f[i], i \in \Z_{2^m}$ is given by the value of $f$ on the binary digits of $i.$
\item
In terms of these two definitions, the Boolean function $f : \Z_2^m \To \Z_2$ is bent if
\begin{align*}
\abs{H_m \V{f}} &= C [1, \ldots, 1]^T.
\end{align*}
for some constant $C.$
\end{enumerate}
 
\end{definition}

The first function, $\sigma_m$ is defined and shown to be bent in \cite{Leo14Constructions}.
We repeat the definition here.

\begin{definition}\label{definition-sign-of-square-function}
We use the basis element selection function $\gamma_m$ of Definition~\ref{definition-gamma} to 
define the \emph{sign-of-square} function $\sigma_m : \Z_2^{2 m} \To \Z_2$ as
\begin{align*}
\sigma_m(i) &:=
\begin{cases}
1 \leftrightarrow \gamma_m(i)^2 = -I
\\
0 \leftrightarrow \gamma_m(i)^2 = I,
\end{cases}
\end{align*}
for all $i$ in $\Z_2^{2 m}$.
\end{definition}
\begin{Remark}
Property 3 from Lemma~\ref{Lemma-representation-properties}
ensures that $\sigma_m$ is well defined.
Also, since each $\gamma_m(i)$ is orthogonal,
$\sigma_m(i) = 1$ if and only if $\gamma_m(i)$ is skew.

From the property of Kronecker products that $(A \otimes B)^T = A^T \otimes B^T,$
it can be shown that $\sigma_m$ can also be calculated from $i \in \Z_2^{2 m}$ as
the parity of the number of occurrences of the bit pair 01 in $i,$
i.e. $\sigma_m(i) = 1$ if and only if the number of 01 pairs is odd.
Alternatively, for $i \in \Z_{2^{2m}},$ $\sigma_m(i) = 1$ if and only if the number of
1 digits in  the base 4 representation of $i$ is odd.
\end{Remark}

The following lemma is proven in \cite{Leo14Constructions}.

\begin{Lemma}\label{Lemma-sigma-m-is-bent}
The function $\sigma_m$ is a bent function on  $\Z_2^{2 m}$.
\end{Lemma}

The basis element selection function $\gamma_m$ also gives rise to a second function,
$\tau_m$ on $\Z_{2^{2 m}}.$
\newpage
\begin{definition}\label{definition-non-diagonal-symmetry-function}
We define the \emph{non-diagonal-symmetry} function $\tau_m$ on $\Z_{2^{2 m}}$ and $\Z_2^{2 m}$
as follows.

For $i$ in $\Z_2^2$:
\begin{align*}
&\tau_1(i) :=
\begin{cases}
1 &\text{if~}i = 10,\text{~so that~}\gamma_1(i) = \pm \oE_2,
\\
0 &\text{otherwise}.
\end{cases}
\end{align*}

For $i$ in $\Z_2^{2 m - 2}$:
\begin{align*}
&\tau_m (00 \odot i) := \tau_{m-1}(i), 
\\
&\tau_m (01 \odot i) := \sigma_{m-1}(i),
\\
&\tau_m (10 \odot i) := \sigma_{m-1}(i) + 1,
\\ 
&\tau_m (11 \odot i) := \tau_{m-1}(i).
\end{align*}
where $\odot$ denotes concatenation of bit vectors, and $\sigma$ is the sign-of-square function, as above.
\end{definition}

It is easy to verify that
$\tau_m(i) = 1$ if and only if $\gamma_m(i)$ is symmetric but not diagonal.
This can be checked directly for $\tau_1.$
For $m > 1$ it results from properties of the Kronecker product of square matrices,
specifically that $(A \otimes B)^T = A^T \otimes B^T,$ and that 
$A \otimes B$ is diagonal if and only if both $A$ and $B$ are diagonal.

The first main result of this paper is the following.
\begin{theorem}
\label{theorem-tau-m-is-bent}
The function $\tau_m$ is a bent function on $\Z_2^{2m}.$
\end{theorem}

The proof of Theorem~\ref{theorem-tau-m-is-bent} uses the following result, 
due to Tokareva \cite{Tok11number}, and stemming from the work of Canteaut, Charpin
and others \cite[Theorem~V.4]{CanCCP01cryptographic}\cite[Theorem~2]{CanC03decomposing}.
The result relies on the following definition.

\begin{definition}
For a bent function $f$ on $\Z_2^m$ the \emph{dual} function $\widetilde{f}$ is given by
\begin{align*}
(H_m [f])_i =: 2^{m/2} (-1)^{\widetilde{f}(i)} .
\end{align*}
\end{definition}

\begin{Lemma}
\cite[Theorem~1]{Tok11number}
\label{Lemma-Tokareva}
If a binary function $f$ on $\Z_2^{2m}$ can be decomposed into four
functions $f_0, f_1, f_2, f_3$ on $\Z_2^{2m-2}$ as
\begin{align*}
f(00 \odot i) =: f_0(i), \quad \quad &f(01 \odot i) =: f_1(i),
\\
f(10 \odot i) =: f_2(i), \quad \quad  &f(11 \odot i) =: f_3(i),
\end{align*}
where all four functions are bent, with dual functions such that
$\widetilde{f}_0 + \widetilde{f}_1 + \widetilde{f}_2 + \widetilde{f}_3 = 1,$
then $f$ is bent.
\end{Lemma}

\begin{proofof}{Theorem~\ref{theorem-tau-m-is-bent}}
In Lemma~\ref{Lemma-Tokareva}, set
$f_0 = f_3 := \tau_{m-1}, f_1 = \sigma_{m-1}, f_2 = \sigma_{m-1} + 1.$
Clearly, $\widetilde{f}_0 = \widetilde{f}_3.$
Also, $\widetilde{f}_2 = \widetilde{f}_1 + 1,$
since $H_{m-1} [f_2] = -H_{m-1} [f_1].$
Therefore 
$\widetilde{f}_0 + \widetilde{f}_1 + \widetilde{f}_2 + \widetilde{f}_3 = 1.$
Thus, these four functions satisfy the premise of Lemma~\ref{Lemma-Tokareva},
as long as both $\sigma_{m-1}$ and $\tau_{m-1}$ are bent. 

It is known that $\sigma_m$ is bent for all $m.$
It is easy to show that $\tau_1$ is bent, directly from its definition.
Therefore $\tau_m$ is bent.
\end{proofof}

\section{Bent functions and Hadamard difference sets}
\label{sec-difference-sets}
The following well known properties of Hada\-mard difference sets and bent functions
are noted in \cite{Leo14Constructions}.

\begin{definition}\label{definition-Hadmard-difference-set} \cite[pp. 10 and 13]{Dil74}.

The $k$-element set $D$ is a $(v,k,\lambda,n)$ \emph{difference set} in an abelian group $G$ of order $v$
if for every non-zero element $g$ in $G,$ 
the equation $g = d_i - d_j$
has exactly $\lambda$ solutions
$(d_i, d_j)$ with $d_i, d_j$ in $D.$
The parameter $n := k-\lambda.$
A $(v,k,\lambda,n)$ difference set with $v=4 n$ is called a 
\emph{Hada\-mard difference set}.
\index{Hadamard difference set}
\end{definition}

\begin{Lemma}
\cite[Remark 2.2.7]{Dil74} \cite{Men62,Rot76}.
\label{Lemma-Hadamard-difference-parameters}
A Hada\-mard difference set has parameters of the form
\begin{align*}
(v,k,\lambda,n) =\ &(4 N^2, 2 N^2 - N, N^2 - N, N^2) \\
   \text{or} \quad &(4 N^2, 2 N^2 + N, N^2 + N, N^2).
\end{align*}
\end{Lemma}

\begin{Lemma}
\cite[Theorem 6.2.2]{Dil74}
\label{Lemma-bent-Hadamard-difference}
The Boolean function $f : \Z_2^m \To \Z_2$ is bent if and only if $D := f^{-1}(1)$
is a Hada\-mard difference set.
\end{Lemma}

Together, these properties, along with Lemma~\ref{Lemma-sigma-m-is-bent} and Theorem~\ref{theorem-tau-m-is-bent}, 
are used here to prove the following result.
\begin{theorem}
\label{theorem-sigma-tau-Hadamard-difference-parameters}
The sets $\sigma_m^{-1}(1)$ and $\tau_m^{-1}(1)$ are both Hadamard difference sets, with the same parameters
\begin{align*}
(v_m,k_m,\lambda_m,n_m) &= (4^m, 2^{2 m - 1} - 2^{m-1}, 2^{2 m - 2} - 2^{m-1}, 2^{2 m - 2}).
\end{align*}
\end{theorem}

\begin{proof}
Both $\sigma_m$ and $\tau_m$ are bent functions, as per
Lemma~\ref{Lemma-sigma-m-is-bent} and Theorem~\ref{theorem-tau-m-is-bent} respectively.
Therefore, by Lemma~\ref{Lemma-bent-Hadamard-difference}, both $\sigma_m^{-1}(1)$ and $\tau_m^{-1}(1)$ are Hadamard difference sets.
In both cases, the relevant abelian group is $\Z_2^{2 m},$ with order $4^m.$
Thus in Lemma~\ref{Lemma-Hadamard-difference-parameters} we must set $N=2^{m-1}$
to obtain that either
\begin{align*}
(v_m,k_m,\lambda_m,n_m) &= (4^m, 2^{2 m - 1} - 2^{m-1}, 2^{2 m - 2} - 2^{m-1}, 2^{2 m - 2})\ \text{or}
\\
(v_m,k_m,\lambda_m,n_m) &= (4^m, 2^{2 m - 1} + 2^{m-1}, 2^{2 m - 2} + 2^{m-1}, 2^{2 m - 2}).
\end{align*}
Since $\sigma_m(i) = 1$ if and only if $\gamma_m(i)$ is skew, and
$\tau_m(i) = 1$ if and only if $\gamma_m(i)$ is symmetric but not diagonal,
not only are these conditions mutually exclusive,
but also, for all $m \geqslant 1,$ the number of $i$ for which $\sigma_m(i) = \tau_m(i) = 0$ is positive.
These are the $i$ for which $\gamma_m(i)$ is diagonal.
Thus $k_m = 2^{2 m - 1} - 2^{m-1}$ rather than $2^{2 m - 1} + 2^{m-1}.$
The result follows immediately.
\qed
\end{proof}
As a check, the parameters $k_m$ can also be calculated directly, using the recursive definitions of each of
$\sigma_m$ and $\tau_m.$

\section{Bent functions and strongly regular graphs}
\label{sec-strongly-regular}
This section examines the relationship between the bent functions $\sigma_m$ and $\tau_m$ and
the subgraphs $\varDelta_m[-1]$ and $\varDelta_m[1]$ from Question~\ref{Question-1}.

First we revise some known properties of Cayley graphs and strongly regular graphs,
as noted in the previous paper on Hada\-mard matrices and Clifford algebras \cite{Leo14Constructions},
including the result of
Bernasconi and Codenotti~\cite{BerC99} on the relationship between bent functions and
strongly regular graphs.

First we recall a special case of the definition of a Cayley graph.
\begin{definition}\label{definition-Cayley-graph}
The \emph{Cayley graph} of a binary function $f : \Z_2^m \To \Z_2$ is 
the undirected graph with adjacency matrix $F$
given by $F_{i,j} = f(g_i + g_j),$ for some ordering $(g_1, g_2, \ldots)$ of $\Z_2^m.$
\end{definition}


The result of
Bernasconi and Codenotti~\cite{BerC99} on the relationship between bent functions and
strongly regular graphs is the following.
\begin{Lemma}\label{Lemma-Cayley-bent-strongly-regular} \cite[Lemma 12]{BerC99}.
The Cayley graph of a bent function on $Z_2^m$ is a strongly regular graph with $\lambda = \mu.$ 
\end{Lemma}

We use this result to examine the graph $\varDelta_m.$
The following two definitions appear in the previous paper \cite{Leo14Constructions}
and are repeated here for completeness.
\begin{definition}\label{definition-delta}
Let $\varDelta_m$ be the graph whose vertices are the $n^2=4^m$ 
canonical basis matrices of the real representation
of the Clifford algebra $\R_{m,m}$,
with each edge having one of two colours, $-1$ (red) and $1$ (blue):
\begin{itemize}
\item 
Matrices $A_j$ and $A_k$ are connected by a red edge if they have disjoint support and are anti-amicable,
i.e. $A_j A_k^{-1}$ is skew.
\item 
Matrices $A_j$ and $A_k$ are connected by a blue edge if they have disjoint support and are amicable,
i.e. $A_j A_k^{-1}$ is symmetric.
\item 
Otherwise there is no edge between $A_j$ and $A_k$.
\end{itemize}
We call this graph the \emph{restricted amicability / anti-amicability graph}
\index{amicability / anti-amicability graph}
of the Clifford algebra $\R_{m,m},$
the restriction being the requirement that an edge only exists for pairs of matrices with disjoint support.
\end{definition}

\begin{definition}\label{definition-red-subgraph}
For a graph $\Gamma$ with edges coloured by -1 (red) and 1 (blue),
$\Gamma[-1]$ denotes the \emph{red subgraph} of $\Gamma,$
the graph containing all of the vertices of $\Gamma,$ and all of the red (-1) coloured edges.
Similarly, $\Gamma[1]$ denotes the \emph{blue subgraph} of $\Gamma.$
\end{definition}

The following theorem is presented in \cite{Leo14Constructions}.
\begin{theorem}\label{th-phi-m-is-strongly-regular}
For all $m \geqslant 1,$
the graph $\varDelta_m[-1]$ is strongly regular, with parameters
$v_m = 4^m,$ $k_m = 2^{2 m - 1} - 2^{m - 1},$ $\lambda_m=\mu_m=2^{2 m - 2} - 2^{m - 1}.$
\end{theorem}
Unfortunately, the proof given there is incomplete,
proving only that $\varDelta_m[-1]$ is strongly regular,
without showing why $k_m = 2^{2 m - 1} - 2^{m - 1}$ and $\lambda_m=\mu_m=2^{2 m - 2} - 2^{m - 1}.$
In this section, we rectify this by proving the following.

\begin{theorem}\label{theorem-twins-are-strongly-regular}
For all $m \geqslant 1,$
both graphs $\varDelta_m[-1]$ and $\varDelta_m[1]$ is strongly regular, with parameters
$v_m = 4^m,$ $k_m = 2^{2 m - 1} - 2^{m - 1},$ $\lambda_m=\mu_m=2^{2 m - 2} - 2^{m - 1}.$
\end{theorem}

\begin{proof}
Since each vertex of $\varDelta_m$ is a canonical basis element of the Clifford algebra $\R_{m,m},$
we can impose the Kronecker product ordering on the vertices, 
labelling each vertex $A$ by $\gamma_m{-1}(A) \in \Z_2^{2m}.$
The label $\kappa_m(a,b)$ of each edge $(\gamma_m(a),\gamma_m(b))$ of $\varDelta_m$ depends on $a+b$
in the following way:
\begin{align*}
\kappa_m(a,b) &:= \tau_m(a+b) - \sigma_m(a+b), \text{~that is,}
\\
\kappa_m(a,b) &=
\begin{cases}
-1, & \sigma_m(a+b) = 1 \quad (\Leftrightarrow \gamma_m(a+b) \text{~is skew}),
\\
0, & \sigma_m(a+b) = \tau_m(a+b) = 0 \quad (\Leftrightarrow \gamma_m(a+b) \text{~is diagonal}),
\\
1, & \tau_m(a+b) = 1 \quad (\Leftrightarrow \gamma_m(a+b) \text{~is symmetric but not diagonal}).
\end{cases}
\end{align*}
Thus $\varDelta_m[-1]$ is isomorphic to the Cayley graph of $\sigma_m$ on $\Z_2^{2m},$
and $\varDelta_m[1]$ is isomorphic to the Cayley graph of $\tau_m$ on $\Z_2^{2m}.$
Since, by Lemma~\ref{Lemma-sigma-m-is-bent} and Theorem~\ref{theorem-tau-m-is-bent},
both $\sigma_m$ and $\tau_m$ are bent functions on $\Z_2^{2m},$ 
Lemma~\ref{Lemma-Cayley-bent-strongly-regular} implies that
both $\varDelta_m[-1]$ and $\varDelta_m[1]$ are strongly regular graphs.

It remains to determine the graph parameters.
Firstly, $v_m$ is the number of vertices, which is $4^m.$

Since $\varDelta_m[-1]$ is regular, we can determine $k_m$ by examining one vertex,
$\gamma_m(0).$
The edges $(\gamma_m(0),\gamma_m(b))$ of $\varDelta_m[-1]$ are those for which $\sigma_m(b)=1,$
that is, the edges where $b$ is in the Hadamard difference set $\sigma_m^{-1}(1).$
Thus, by Theorem~\ref{theorem-sigma-tau-Hadamard-difference-parameters}, $k_m = 2 N^2 - N = 2^{2 m - 1} - 2^{m-1},$
where $N=2^{m-1}.$

Since $\varDelta_m[-1]$ is a strongly regular graph, it holds that
\begin{align*}
(v_m - k_m - 1)\mu_m &= k_m (k_m - 1 - \lambda_m)  
\end{align*}
\cite[p. 158]{Sei79}
and hence, since $\lambda_m=\mu_m,$ we must have $(v_m-1)\lambda_m = k_m(k_m-1).$
We now note that
\begin{align*}
k_m (k_m-1) &= (2 N^2 - N)(2 N^2 - N - 1)
= (v_m-1)(2^{2m - 2} - 2^{m-1}),
\end{align*}
so that $\lambda_m=\mu_m=2^{2 m - 2} - 2^{m-1}.$

Running through these arguments again, with $\varDelta_m[1]$ substituted for $\varDelta_m[-1]$
and $\tau_m$ substituted for $\sigma_m,$ yields the same parameters for $\varDelta_m[1].$
\qed   
\end{proof}

\begin{Remark}
A more elementary derivation of the value of $\lambda_m$ for $\varDelta_m[-1]$ follows.

There are $k_m(k_m-1)$ ordered pairs $(a,b)$ with $a \neq b$ and $\sigma_m(a)=\sigma_m(b)=1.$
Since $k_m (k_m-1)= (N^2-N)(4N^2-1),$
this gives exactly $N^2-N = 2^{2m - 2} - 2^{m-1}$ ordered pairs for each of other $4^m-1$ vertices of $\varDelta_m[-1].$

Also, considering that $\sigma_m^{-1}(1)$ is a Hadamard difference set,
and for $c \in \Z_2^{2m},$ $c \neq 0,$ consider one of the pairs $(a,b)$ such that
$\sigma_m(a)=\sigma_m(b)=1$ and $c=a+b.$
Thus $b=a+c$ and $\sigma_m(a)=\sigma_m(a+c)=1.$
Therefore, the graph $\varDelta_m[-1]$ contains the edges 
$(\gamma_m(0),\gamma_m(a)),$ $(\gamma_m(0),\gamma_m(b)),$ $(\gamma_m(c),\gamma_m(a)),$ and $(\gamma_m(c),\gamma_m(b)).$
\newpage
Thus, in the graph $\varDelta_m[-1],$ the vertices $\gamma_m(0)$ and $\gamma_m(c)$ have the two vertices
$\gamma_m(a)$ and $\gamma_m(b)$ in common.
This is true whether or not there is an edge between $\gamma_m(0)$ and $\gamma_m(c).$
The pair $(b,a)$ yields the same four edges.
Running through all such pairs $(a,b)$ 
and using Theorem~\ref{theorem-sigma-tau-Hadamard-difference-parameters} again,
we see that $\lambda_m=\mu_m=2N^2-N=2^{2 m - 2} - 2^{m-1}.$
\end{Remark}

\section{Other necessary conditions}
\label{sec-other}
This section examines two other necessary conditions for the existence of an automorphism of $\varDelta_m$ that swaps
$\varDelta_m[-1]$ with $\varDelta_m[1].$
The first condition follows.
\begin{theorem}
If an automorphism $\theta : \varDelta_m \To \varDelta_m$ exists that swaps $\varDelta_m[-1]$ with $\varDelta_m[1],$
then there is an automorphism $\Theta : \varDelta_m \To \varDelta_m$ that also swaps $\varDelta_m[-1]$ with $\varDelta_m[1],$
leaving $\gamma_m(0)$ fixed.
\end{theorem}
\begin{proof}
For the purposes of this proof, assume the Kronecker product ordering of the canonical basis elements of $\R_{m,m}$
and define the one-to-one mapping $\phi : \Z_2^{2m} \To \Z_2^{2m}$ such that $\theta(\gamma_m(a)) = \gamma_m(\phi(a))$
for all $a \in \Z_2^{2m}.$
The condition that $\theta$ swaps $\varDelta_m[-1]$ with $\varDelta_m[1]$ is equivalent to the condition
\begin{align*}
\kappa_m(\phi(a)+\phi(b)) &= -\kappa_m(a+b),
\end{align*}
where $\kappa_m$ is as defined in the proof of Theorem~\ref{theorem-twins-are-strongly-regular} above.

Let $\Phi(a) := \phi(a)+\phi(0)$ for all $a \in Z_2^{2m}.$
Then $\Phi(a)+\Phi(b) = \phi(a)+\phi(b)$ for all $a,b \in Z_2^{2m},$
and therefore
\begin{align*}
\kappa_m(\Phi(a)+\Phi(b)) = \kappa_m(\phi(a)+\phi(b)) = -\kappa_m(a+b).
\end{align*}
Now define $\Theta : \varDelta_m \To \varDelta_m$ such that  $\Theta(\gamma_m(a)) = \gamma_m(\Phi(a))$
for all $a \in \Z_2^{2m}.$
\qed
\end{proof}

The second condition is simply to note that if $\theta$ swaps $\varDelta_m[-1]$ with $\varDelta_m[1],$
then for any induced subgraph $\Gamma \subset \varDelta_m$ and its image $\theta(\Gamma),$
the corresponding edges $(A,B)$ and $(\theta(A),\theta(B))$ will also have swapped colours.

These two conditions were used to design a backtracking search algorithm
to find an automorphism that satisfies Question~\ref{Question-1} or rule out its existence.
Two implementations of the search algorithm were coded: one using Python, and a faster implementation using Cython.
The source code is available on GitHub \cite{Leo15Hadamard}.
Running the search confirms the existence of an automorphism for $m=1,2,$ and $3,$ but rules it out for $m=4.$
On an Intel\textregistered ~Core\texttrademark ~i7 CPU 870 @ 2.93GHz, the Cython implementation of search for $m=4$ takes about 15 hours to run.

Since this paper was submitted, the author has found a simple proof that an automorphism 
satisfying Question~\ref{Question-1} does not exist for $m > 4$: See arXiv:1504.02827 [math.CO].
\newpage
\subsection*{Acknowledgements.}
This work was first presented at the  
Workshop on Algebraic Design Theory and Hadamard Matrices (ADTHM) 2014,
in honour of the 70th birthday of Hadi Kharaghani.
Thanks to Robert Craigen, and William Martin for valuable discussions,
and again to Robert Craigen for presenting Questions 1 and 2 at 
the workshop on ``Algebraic design theory with Hada\-mard matrices'' in Banff in July 2014.
Thanks also to the Mathematical Sciences Institute at The Australian National University 
for the author's Visiting Fellowship during 2014.
Finally, thanks to the anonymous reviewer whose comments have helped to improve this paper.

\end{document}